\documentclass[12pt]{amsart}

\usepackage{enumerate, amsthm}

\newtheorem{theorem}{Theorem}[section]
\newtheorem{lemma}[theorem]{Lemma}
\newtheorem{prop}[theorem]{Proposition}
\newtheorem{cor}[theorem]{Corollary}

\theoremstyle{remark}
\newtheorem{definition}[theorem]{Definition}

\DeclareMathOperator{\dist}{dist} 

\newcommand{\bQ}{\mathbb{Q}}
\newcommand{\bN}{\mathbb{N}}
\newcommand{\bP}{\mathbb{P}}
\newcommand{\bR}{\mathbb{R}}
\newcommand{\bZ}{\mathbb{Z}}
\newcommand{\cC}{\mathcal{C}}

\newcommand{\cO}{\mathcal{O}}

\newcommand{\Cbar}{\bar{\cC}}
\newcommand{\Cli}{\cC^{\textit{li}}}

\newcommand{\et}{\quad\text{and}\quad}
\newcommand{\GL}{\mathrm{GL}}
\newcommand{\Lbar}{\bar{L}}

\newcommand{\ssi}{\quad\Longleftrightarrow\quad}

\newcommand{\ua}{\mathbf{a}}
\newcommand{\ur}{\mathbf{r}}
\newcommand{\uu}{\mathbf{u}}
\newcommand{\uv}{\mathbf{v}}
\newcommand{\uw}{\mathbf{w}}
\newcommand{\ux}{\mathbf{x}}
\newcommand{\uxi}{\underline{\xi}}
\newcommand{\uXi}{\underline{\Xi}}
\newcommand{\uy}{\mathbf{y}}
\newcommand{\uz}{\mathbf{z}}

\numberwithin{equation}{section}
\begin{document}

\baselineskip=15pt

\title[Rational approximation on conics]
{Rational approximation to real points on conics}
\author{Damien ROY}
\address{
   D\'epartement de Math\'ematiques\\
   Universit\'e d'Ottawa\\
   585 King Edward\\
   Ottawa, Ontario K1N 6N5, Canada}
\email[Damien Roy]{droy@uottawa.ca}
\thanks{Research partially supported by NSERC}


\begin{abstract}
A point $(\xi_1,\xi_2)$ with coordinates in a subfield of $\bR$ of transcendence degree one over $\bQ$, with $1,\xi_1,\xi_2$ linearly independent over $\bQ$, may have a uniform exponent of approximation by elements of $\bQ^2$ that is strictly larger than the lower bound $1/2$ given by Dirichlet's box principle. This appeared as a surprise, in connection to work of Davenport and Schmidt, for points of the parabola $\{(\xi,\xi^2) \,;\, \xi\in\bR\}$.  The goal of this paper is to show that this phenomenon extends to all real conics defined over $\bQ$, and that the largest exponent of approximation achieved by points of these curves satisfying the above condition of linear independence is always the same, independently of the curve, namely $1/\gamma\cong 0.618$ where $\gamma$ denotes the golden ratio.
\par
\medskip
\noindent
{\sc R\'esum\'e.}
Un point $(\xi_1,\xi_2)$ \`a coordonn\'ees dans un sous-corps de $\bR$ de degr\'e de transcendence un sur $\bQ$, avec $1,\xi_1,\xi_2$ lin\'eairement ind\'ependants sur $\bQ$, peut admettre un exposant d'approximation uniforme par les \'el\'ements de $\bQ^2$ qui soit strictement plus grand que la borne inf\'erieure $1/2$ que garantit le principe des tiroirs de Dirichlet.  Ce fait inattendu est apparu, en lien avec des travaux de Davenport et Schmidt, pour les points de la parabole $\{(\xi,\xi^2) \,;\, \xi\in\bR\}$.  Le but de cet article est de montrer que ce ph\'enom\`ene s'\'etend \`a toutes les coniques r\'eelles d\'efinies sur $\bQ$, et que le plus grand exposant d'approximation atteint par les points de ces courbes, sujets \`a la condition d'ind\'ependance lin\'eaire mentionn\'ee plus t\^ot, est toujours le m\^eme, ind\'ependamment de la courbe, \`a savoir $1/\gamma\cong 0.618$ o\`u $\gamma$ d\'esigne le nombre d'or.
\end{abstract}

\maketitle

%
%

\section{Introduction}
\label{sec:intro}

Let $n$ be a positive integer and let $\uxi=(\xi_1,\dots,\xi_n)\in \bR^n$.  The \emph{uniform exponent of approximation to $\uxi$ by rational points}, denoted $\lambda(\uxi)$, is defined as the supremum of all real numbers $\lambda$ for which the system of inequalities
\begin{equation}
 \label{intro:eq:lambda}
 |x_0|\le X, \quad \max_{1\le i\le n} |x_0\xi_i-x_i| \le X^{-\lambda}
\end{equation}
admits a non-zero solution $\ux=(x_0,x_1,\dots,x_n)\in \bZ^{n+1}$ for each sufficiently large real number $X>1$.  It is one of the classical ways of measuring how well $\uxi$ can be approximated by elements of $\bQ^n$, because each solution of \eqref{intro:eq:lambda} with $x_0\neq 0$ provides a rational point $\ur=(x_1/x_0,\dots,x_n/x_0)$ with denominator dividing $x_0$ such that $\|\uxi-\ur\| \le |x_0|^{-\lambda-1}$, where the symbol $\|\ \|$ stands for the maximum norm.  We call it a ``uniform exponent'' following the terminology of Y.~Bugeaud and M.~Laurent in \cite[\S1]{BL} because we require a solution of \eqref{intro:eq:lambda} for each sufficiently large $X$ (but note that our notation is slightly different as they denote it $\hat{\lambda}(\uxi)$).  This exponent depends only on the $\bQ$-vector subspace of $\bR$ spanned by $1,\xi_1,\dots,\xi_n$ and so, by a result of Dirichlet \cite[Chapter II, Theorem 1A]{Sc}, it satisfies $\lambda(\uxi)\ge 1/(s-1)$ where $s\ge 1$ denotes the dimension of that subspace.  In particular we have $\lambda(\uxi)=\infty$ when $\uxi\in\bQ^n$, while it is easily shown that $\lambda(\uxi)\le 1$ when $\uxi\notin\bQ^n$ (see for example \cite[Prop.~2.1]{BL}).

In their seminal work \cite{DS}, H.~Davenport and W.~M.~Schmidt determine an upper bound $\lambda_n$, depending only on $n$, for $\lambda(\xi,\xi^2,\dots,\xi^n)$ where $\xi$ runs through all real numbers such that $1,\xi,\dots,\xi^n$ are linearly independent over $\bQ$, a condition which amounts to asking that $\xi$ is not algebraic over $\bQ$ of degree $n$ or less.  Using geometry of numbers, they deduce from this a result of approximation to such $\xi$ by algebraic integers of degree at most $n+1$.  In particular they prove that $\lambda(\xi,\xi^2) \le \lambda_2:=1/\gamma\cong 0.618$ for each non-quadratic irrational real number $\xi$, where $\gamma=(1+\sqrt{5})/2$ denotes the golden ratio.  It is shown in \cite{Rnote,RcubicI} that this upper bound is best possible and, in \cite{RcubicII}, that the corresponding result of approximation by algebraic integers of degree at most $3$ is also best possible.  For $n\ge 3$, no optimal value is known for $\lambda_n$.  At present the best known upper bounds are $\lambda_3 \le (1 + 2\gamma - \sqrt{1+4\gamma^2})/2 \cong 0.4245$ (see \cite{Racta}) and $\lambda_n\le 1/\lceil n/2\rceil$ for $n\ge 4$ (see \cite{La}).

As a matter of approaching this problem from a different angle, we propose to extend it to the following setting.

\begin{definition}
Let $\cC$ be a closed algebraic subset of $\bR^n$ of dimension $1$ defined over $\bQ$, irreducible over $\bQ$, and not contained in any proper affine linear subspace of $\bR^n$ defined over $\bQ$.  Then, we put
$
 \lambda(\cC)=\sup\{\lambda(\uxi)\,;\, \uxi\in \Cli\}
$
where $\Cli$ denotes the set of points $\uxi=(\xi_1,\dots,\xi_n)\in\cC$ such that $1,\xi_1,\dots,\xi_n$ are linearly independent over $\bQ$.
\end{definition}

Equivalently, such a curve may be described as the Zariski closure over $\bQ$ in $\bR^n$ of a point $\uxi \in \bR^n$ whose coordinates $\xi_1,\dots,\xi_n$ together with $1$ are linearly independent over $\bQ$ and generate over $\bQ$ a subfield of $\bR$ of transcendence degree one.  In particular $\Cli$ is not empty as it contains that point.  From the point of view of metrical number theory the situation is simple since, for the relative Lebesgue measure, almost all points $\uxi$ of $\cC$ have $\lambda(\uxi)=1/n$ (see \cite{K}).  Of special interest is the curve $\cC_n:=\{(\xi,\xi^2,\dots,\xi^n)\,;\, \xi\in \bR\}$ for any $n\ge 2$.  As mentioned above, we have $\lambda(\cC_2)=1/\gamma$ and the problem remains to compute $\lambda(\cC_n)$ for $n\ge 3$.  In this paper, we look at the case of conics in $\bR^2$ and prove the following result.

\begin{theorem}
\label{intro:thm}
Let $\cC$ be a closed algebraic subset of $\bR^2$ of dimension $1$ and degree $2$.  Suppose that $\cC$ is defined over $\bQ$ and irreducible over $\bQ$.  Then, we have $\lambda(\cC)=1/\gamma$.  Moreover, the set of points $\uxi\in\Cli$ with $\lambda(\uxi)=1/\gamma$ is countably infinite.
\end{theorem}

Here the \emph{degree} of $\cC$ simply refers to the degree of the irreducible polynomial of $\bQ[x_1,x_2]$ defining it.  The curve $\cC_2$ is the parabola of equation $x_2-x_1^2=0$ but, as we will see, other curves are easier to deal with, for example the curve defined by $x_1^2-2=0$ which consists of the pair of vertical lines $\{\pm\sqrt{2}\}\times \bR$.  Note that, for the latter curve, Theorem \ref{intro:thm} simply says that any $\xi\in\bR\setminus\bQ(\sqrt{2})$ has $\lambda(\sqrt{2},\xi)\le 1/\gamma$, with equality defining a denumerable subset of $\bR \setminus \bQ(\sqrt{2})$.  Our main result in the next section provides a slightly finer result.

In \cite{LR}, it is shown that the cubic $\cC$ defined by $x_2-x_1^3=0$ has $\lambda(\cC) \le 2(9+\sqrt{11})/35 \cong 0.7038$, but the case of the line $\sqrt[3]{2}\times \bR$ should be simpler to solve and could give ideas to determine the precise value of $\lambda(\cC)$ for that cubic $\cC$.  Similarly, looking at lines $(\omega_2,\dots,\omega_n)\times\bR$ where $(1,\omega_2,\dots,\omega_n)$ is a basis over $\bQ$ of a number field  of degree $n$ could provide new ideas to compute $\lambda(\cC_n)$.

This paper is organized as follows. In the next section, we state a slightly stronger result in projective setting and note that, for curves $\cC$ which are irreducible over $\bR$ and contain at least one rational point, the proof simply reduces to the known case of the parabola $\cC_2$.   In Section 3, we prove the inequality $\lambda(\cC)\le 1/\gamma$ for the remaining curves $\cC$ by an adaptation of the original argument of Davenport and Schmidt in \cite[\S3]{DS}.  However, the fact that these curves have at most one rational point brings a notable simplification in the proof.  In Section 4, we adapt the arguments of \cite[\S5]{RcubicI} to establish a certain rigidity property for the sequence of minimal points attached to points $\uxi\in\Cli$ with $\lambda(\uxi)=1/\gamma$, and deduce from it that the set of these points $\uxi$ is at most countable. We conclude in Section 5, with the most delicate part, namely the existence of infinitely many points $\uxi\in\Cli$ having exponent $1/\gamma$.

%
%

\section{The main result in projective framework}
\label{sec:proj}

For each $n\ge 2$, we endow $\bR^n$ with the maximum norm, and identify its exterior square $\bigwedge^2\bR^n$ with $\bR^{n(n-1)/2}$ via an ordering of the Pl\"ucker coordinates.  In particular, when $n=3$, we define the wedge product of two vectors in $\bR^3$ as their usual cross-product.  We first introduce finer notions of Diophantine approximation in the projective context.

Let $\Xi\in\bP^n(\bR)$ and let $\uXi=(\xi_0,\dots,\xi_n)$ be a representative of $\Xi$ in $\bR^{n+1}$.  We say that a real number $\lambda\ge 0$ is an \emph{exponent of approximation} to $\Xi$ if there exists a constant $c=c_1(\uXi)$ such that the conditions
\[
 \|\ux\| \le X \et \|\ux\wedge\uXi\| \le cX^{-\lambda}
\]
admit a non-zero solution $\ux\in\bZ^{n+1}$ for each sufficiently large real number $X$.  We say that $\lambda$ is a \emph{strict exponent of approximation} to $\Xi$ if moreover there exists a constant $c=c_2(\uXi)>0$ such that the same conditions admit no non-zero solution $\ux\in\bZ^{n+1}$ for arbitrarily large values of $X$.  Both properties are independent of the choice of the representative $\uXi$, and we define $\lambda(\Xi)$ as the supremum of all exponents of approximations to $\Xi$.  Clearly, when $\lambda$ is a strict exponent of approximation to $\Xi$, we have $\lambda(\Xi)=\lambda$.

Let $T\colon \bQ^{n+1}\to \bQ^{n+1}$ be an invertible $\bQ$-linear map.  It extends uniquely to a $\bR$-linear automorphism of $\bR^{n+1}$ and then to an automorphism of $\bP^n(\bR)$.  This defines an action of $\GL_{n+1}(\bQ)$ on $\bP^n(\bR)$. Moreover, upon choosing an integer $m\ge 1$ such that $mT(\bZ^{n+1})\subseteq \bZ^{n+1}$, any non-zero point $\ux\in\bZ^{n+1}$ gives rise to a non-zero point $\uy=mT(\ux)\in\bZ^{n+1}$ satisfying
\[
 \|\uy\| \le c_T \|\ux\| \et \|\uy\wedge T(\uXi)\| \le c_T \|\ux\wedge\Xi\|
\]
for a constant $c_T>0$ depending only on $T$. Combined with the above definitions, this yields the following invariance property.

\begin{lemma}
\label{proj:lemma:T}
Let $\Xi\in\bP^n(\bR)$ and $T\in\GL_{n+1}(\bQ)$.  Then we have $\lambda(\Xi)=\lambda(T(\Xi))$.  More precisely a real number $\lambda\ge 0$ is an exponent of approximation to $\Xi$, respectively a strict exponent of approximation to $\Xi$, if and only if it is an exponent of approximation to $T(\Xi)$, respectively a strict exponent of approximation to $T(\Xi)$.
\end{lemma}

We also have a natural embedding of $\bR^n$ into $\bP^n(\bR)$, sending a point $\uxi=(\xi_1,\dots,\xi_n)$ to $(1:\uxi):=(1:\xi_1:\cdots:\xi_n)$.  Identifying $\bR^n$ with its image in $\bP^n(\bR)$, the above notions of exponent of approximation and strict exponent of approximation carry back to points of $\bR^n$.  The next lemma, whose proof is left to the reader, shows how they translate in this context and shows moreover that $\lambda(\uxi)=\lambda(1:\uxi)$, thus leaving no ambiguity as to the value of $\lambda(\uxi)$.

\begin{lemma}
\label{proj:lemma:affine_exponents}
Let $\uxi=(\xi_1,\dots,\xi_n)\in\bR^n$.
\begin{itemize}
 \item[(i)] A real number $\lambda\ge 0$ is an exponent of approximation to $(1:\uxi)$ if and only if there exists a constant $c=c_1(\uxi)$ such that the conditions
     \begin{equation*}
       \label{proj:eq2}
       |x_0| \le X \et \max_{1\le i\le n} |x_0\xi_i-x_i| \le cX^{-\lambda}
     \end{equation*}
     admit a non-zero solution $\ux=(x_0,\dots,x_n)\in\bZ^{n+1}$ for each sufficiently large $X$.
 \item[(ii)] It is a strict exponent of approximation to $(1:\uxi)$ if and only if there also exists a constant $c=c_2(\uxi)>0$ such that the above conditions admit no non-zero integer solution for arbitrarily large values of $X$.
\end{itemize}
Finally, we have $\lambda(\uxi)=\lambda(1:\uxi)$.
\end{lemma}

Our main result is the following strengthening of Theorem \ref{intro:thm}.

\begin{theorem}
\label{proj:thm}
Let $\varphi$ be a homogeneous polynomial of degree $2$ in $\bQ[x_0,x_1,x_2]$.  Suppose that $\varphi$ is irreducible over $\bQ$ and that its set of zeros $\cC$ in $\bP^2(\bR)$ consists of at least two points.
\begin{enumerate}
 \item[(i)] For each point $\Xi\in \cC$ having $\bQ$-linearly independent homogeneous coordinates, the number $1/\gamma$ is at best a strict exponent of approximation to $\Xi$: if it is an exponent of approximation to $\Xi$, it is a strict one.
 \item[(ii)] There are infinitely many points $\Xi\in \cC$ which have $\bQ$-linearly independent homogeneous coordinates and for which $1/\gamma$ is an exponent of approximation.
 \item[(iii)] There exists a positive $\epsilon$, independent of $\varphi$, such that the set of points $\Xi\in\cC$ with $\lambda(\Xi) > 1/\gamma-\epsilon$ is countable.
\end{enumerate}
\end{theorem}

To show that this implies Theorem \ref{intro:thm}, let $\cC$ be as in latter statement. Then, the Zariski closure $\Cbar$ of $\cC$ in $\bP^2(\bR)$ is infinite and is the zero set of an irreducible homogeneous polynomial of degree $2$ in $\bQ[x_0,x_1,x_2]$.  Moreover, $\Cli$ identifies with the set of elements of $\Cbar$ with $\bQ$-linearly independent homogeneous coordinates.  So, if we admit the above theorem, then, in view of Lemma \ref{proj:lemma:affine_exponents}, Part (i) implies that $\lambda(\cC)\le 1/\gamma$, Part (ii) shows that there are infinitely many $\uxi\in\Cli$ with $\lambda(\uxi)=1/\gamma$, and Part (iii) shows that the set of points $\uxi\in\cC$ with $\lambda(\uxi)>1/\gamma-\epsilon$ is countable. Altogether, this proves Theorem \ref{intro:thm}.

The proof of Part (iii) in Section \ref{sec:iii} will show that one can take $\epsilon=0.005$ but the optimal value for $\epsilon$ is probably much larger.  In connection to (iii), we also note that the set of elements of $\cC$ with $\bQ$-linearly dependent homogeneous coordinates is at most countable because each such point belongs to a proper linear subspace of $\bP^2(\bR)$ defined over $\bQ$, there are countably many such subspaces, and each of them meets $\cC$ in at most two points.  So, in order to prove (iii), we may restrict to the points of $\cC$ with $\bQ$-linearly independent homogeneous coordinates.

Lemma \ref{proj:lemma:T} implies that, if Theorem \ref{proj:thm} holds true for a form $\varphi$, then it also holds for $\mu(\varphi\circ T)$ for any $T\in\GL_3(\bQ)$ and any $\mu\in\bQ^*$.  Thus the next lemma reduces the proof of the theorem to forms of special types.

\begin{lemma}
\label{proj:lemma:reduction}
Let $\varphi$ be an irreducible homogeneous polynomial of $\bQ[x_0,x_1,x_2]$ of degree $2$ which admits at least two zeros in $\bP^2(\bR)$.
\begin{itemize}
 \item[(i)] If $\varphi$ is irreducible over $\bR$ and admits at least one zero in $\bP^2(\bQ)$, then there exist $\mu\in\bQ^*$ and $T\in\GL_3(\bQ)$ such that $\mu(\varphi\circ T)(x_0,x_1,x_2)=x_0x_2-x_1^2$.
 \item[(ii)] If $\varphi$ is not irreducible over $\bR$, then it admits exactly one zero in $\bP^2(\bQ)$ and there exist $\mu\in\bQ^*$ and $T\in\GL_3(\bQ)$ such that $\mu(\varphi\circ T)(x_0,x_1,x_2)=x_0^2-bx_1^2$ for some square-free integer $b>1$.
 \item[(iii)] If $\varphi$ has no zero in $\bP^2(\bQ)$, then there exist $\mu\in\bQ^*$ and $T\in\GL_3(\bQ)$ such that $\mu(\varphi\circ T)(x_0,x_1,x_2)=x_0^2-bx_1^2-cx_2^2$ for some square-free integers $b>1$ and $c>1$.
\end{itemize}
\end{lemma}

\begin{proof} We view $(\bQ^3,\varphi)$ as a quadratic space.  We denote by $K$ its kernel, and by $\Phi$ the unique symmetric bilinear form such that $\Phi(\ux,\ux)=2\varphi(\ux)$.

Suppose first that $K\neq \{0\}$.  Then, by a change of variables over $\bQ$, we can bring $\varphi$ to a diagonal form $rx_0^2+sx_1^2$ with $r,s\in\bQ$.  We have $rs\neq 0$ since $\varphi$ is irreducible over $\bQ$, and furthermore $rs<0$ since otherwise the point $(0:0:1)$ would be the only zero of $\varphi$ in $\bP^2(\bR)$.  Thus, $\varphi$ is not irreducible over $\bR$, and $\dim_\bQ K=1$.

In the case (i), the above observation shows that $\bQ^3$ is non-degenerate.  Then, since $\varphi$ has a zero in $\bP^2(\bQ)$, the space $\bQ^3$ decomposes as the orthogonal direct sum of a hyperbolic plane $H$ and a non-degenerate line $P$.  We choose bases $\{\uv_0,\uv_2\}$ for $H$ and $\{\uv_1\}$ for $P$ such that $\varphi(\uv_0)=\varphi(\uv_2)=0$ and $\Phi(\uv_0,\uv_2)=-\varphi(\uv_1)$.  Then $\mu=-1/\varphi(\uv_1)$ and the linear map $T\colon\bQ^3\to\bQ^3$ sending the canonical basis of $\bQ^3$ to $(\uv_0,\uv_1,\uv_2)$ have the property stated in (i).

In the case (iii), we have $K=\{0\}$ and so we can write $\bQ^3$ as an orthogonal direct sum of one-dimensional non-degenerate subspaces $P_0$, $P_1$ and $P_2$.  We order them so that the non-zero values of $\varphi$ on $P_0$ have opposite sign to those on $P_1$ and $P_2$.  This is possible since $\varphi$ is indefinite.  Let $\{\uv_0\}$ be a basis of $P_0$ and put $\mu=1/\varphi(\uv_0)$.  For $i=1,2$, we can choose a basis $\{\uv_i\}$ of $P_i$ such that $\mu\varphi(\uv_i)$ is a square-free integer.  Then $\mu$ and the linear map $T\colon\bQ^3\to\bQ^3$ sending the canonical basis of $\bQ^3$ to $(\uv_0,\uv_1,\uv_2)$ have the property stated in (iii).

In the case (ii), the form $\varphi$ factors over a quadratic extension $\bQ(\sqrt{d})$ of $\bQ$ as a product $\varphi(\ux)=\rho L(\ux)\Lbar(\ux)$ where $L$ is a linear form, $\Lbar$ its conjugate over $\bQ$, and $\rho\in\bQ^*$.  As $\varphi$ is irreducible over $\bQ$, the linear forms $L$ and $\Lbar$ are not multiple of each other.  Moreover, for a point $\ua\in\bQ^3$, we have
\[
 \varphi(\ua)=0 \ssi L(\ua)=\Lbar(\ua)=0 \ssi (L+\Lbar)(\ua)=\sqrt{d}(L-\Lbar)(\ua)=0.
\]
Since $L+\Lbar$ and $\sqrt{d}(L-\Lbar)$ are linearly independent forms with coefficients in $\bQ$, this means that the zero set of $\varphi$ in $\bQ^3$ is a line, and so $\varphi$ has a unique zero in $\bP^2(\bQ)$.  As $\Phi(\ux,\uy) = \rho L(\ux)\Lbar(\uy) + \rho\Lbar(\ux)L(\uy)$, this line is contained in the kernel $K$ of $\varphi$, and so is equal to $K$.  By an earlier observation, this means that, by a change of variables over $\bQ$, we may bring $\varphi$ to a diagonal form $rx_0^2+sx_1^2$ with $r,s\in\bQ$, $rs<0$.  We may further choose $r$ and $s$ so that $-s/r$ is a square-free integer $b>0$.  Then, the same change of variables brings $r^{-1}\varphi$ to $x_0^2-bx_1^2$.  Finally, we have $b\neq 1$ since $\varphi$ is irreducible over $\bQ$.
\end{proof}

%
%

\section{Proof of the first part of the main theorem}
\label{sec:i}

Let $\varphi$ and $\cC$ be as in the statement of Theorem \ref{proj:thm}.  Suppose first that $\varphi$ is irreducible over $\bR$ and that $\cC\cap\bP^2(\bQ)\neq \emptyset$.  Then, by Lemma \ref{proj:lemma:reduction}, there exists $T\in\GL_3(\bQ)$ such that $T^{-1}(\cC)$ is the zero-set in $\bP^2(\bR)$ of the polynomial $x_0x_2-x_1^2$.  Let $\Xi$ be a point of $\cC$ with $\bQ$-linearly independent homogeneous coordinates.  Its image $T^{-1}(\Xi)$ has homogeneous coordinates $(1:\xi:\xi^2)$, for some irrational non-quadratic $\xi\in\bR$.  Then, by \cite[Theorem 1a]{DS}, the number $1/\gamma$ is at best a strict exponent of approximation to $T^{-1}(\Xi)$, and, by Lemma \ref{proj:lemma:T}, the same applies to $\Xi$.  This proves Part (i) of the theorem in that case.

Otherwise, Lemma \ref{proj:lemma:reduction} shows that $\varphi$ has at most one zero in $\bP^2(\bQ)$.  Taking advantage of the major simplification that this entails, we proceed as Davenport and Schmidt in \cite[\S3]{DS}.  We fix a point $\Xi\in\cC$ with $\bQ$-linearly independent homogeneous coordinates $(1:\xi_1:\xi_2)$ and an exponent of approximation $\lambda\ge 1/2$ for $\Xi$.  Then, by Lemma \ref{proj:lemma:affine_exponents}, there exists a constant $c>0$ such that, for each sufficiently large $X$, the system
\begin{equation}
 \label{i:eq:L}
 |x_0|\le X, \quad L(\ux):=\max\{|x_0\xi_1-x_1|,|x_0\xi_2-x_2|\} \le cX^{-\lambda}
\end{equation}
has a non-zero solution $\ux=(x_0,x_1,x_2)\in\bZ^3$.  To prove Part (i) of Theorem \ref{proj:thm}, we simply need to show that $\lambda\le 1/\gamma$ and that, when $\lambda=1/\gamma$, the constant $c$ cannot be chosen arbitrarily small.

To this end, we first note that there exists a sequence of points $(\ux_i)_{i\ge 1}$ in $\bZ^3$ such that
\begin{itemize}
\item[(a)] their first coordinates $X_i$ form an increasing sequence $1 \le X_1<X_2<X_3<\cdots$,
\item[(b)] the quantities $L_i:=L(\ux_i)$ form a decreasing sequence $1>L_1>L_2>L_3>\cdots$,
\item[(c)] for each $\ux=(x_0,x_1,x_2)\in\bZ^3$ and each $i\ge 1$ with $|x_0| < X_{i+1}$, we have $L(\ux)\ge L_i$.
\end{itemize}
Then, each $\ux_i$ is a \emph{primitive} point of $\bZ^3$, by which we mean that the gcd of its coordinates is $1$.  Moreover, the hypothesis that \eqref{i:eq:L} has a solution for each large enough $X$ implies that
\begin{equation}
 \label{i:eq:Lbis}
 L_i \le cX_{i+1}^{-\lambda}
\end{equation}
for each sufficiently large $i$, say for all $i\ge i_0$.  Since $\varphi$ has at most one zero in $\bP^2(\bQ)$, we may further assume that $\varphi(\ux_i)\neq 0$ for each $i\ge i_0$.  Then, upon normalizing $\varphi$ so that it has integer coefficients, we conclude that $|\varphi(\ux_i)|\ge 1$ for the same values of $i$.

Put $\uXi=(1,\xi_1,\xi_2)\in\bQ^3$, and let $\Phi$ denote the symmetric bilinear form for which $\Phi(\ux,\ux)=2\varphi(\ux)$.  Then, upon writing $\ux_i=X_i\uXi+\Delta_i$ and noting that $\varphi(\uXi)=0$, we find
\begin{equation}
 \label{i:eq:varphi}
 \varphi(\ux_i) = X_i\Phi(\uXi,\Delta_i)+\varphi(\Delta_i).
\end{equation}
As $\|\Delta_i\|=L_i$, this yields $|\varphi(\ux_i)|\le c_1X_iL_i$ for a constant $c_1=c_1(\varphi,\uXi)>0$. Using \eqref{i:eq:Lbis}, we conclude that, for each $i\ge i_0$, we have $1\le |\varphi(\ux_i)| \le cc_1X_iX_{i+1}^{-\lambda}$, and so
\begin{equation}
 \label{i:eq:Xip}
 X_{i+1}^{\lambda}\le cc_1 X_i.
\end{equation}

We also note that there are infinitely many values of $i>i_0$ for which $\ux_{i-1}$, $\ux_i$ and $\ux_{i+1}$ are linearly independent.  For otherwise, all points $\ux_i$ with $i$ large enough would lie in a two dimensional subspace $V$ of $\bR^3$ defined over $\bQ$.  As the products $X_i^{-1}\ux_i$ converge to $\uXi$ when $i\to\infty$, this would imply that $\uXi\in V$, in contradiction with the hypothesis that $\uXi$ has $\bQ$-linearly independent coordinates.  Let $I$ denote the set of these indices $i$.

For $i\in I$, the integer $\det(\ux_{i-1},\ux_i,\ux_{i+1})$ is non-zero and \cite[Lemma 4]{DS} yields
\[
 1 \le |\det(\ux_{i-1},\ux_i,\ux_{i+1})|
   \le 6 X_{i+1}L_iL_{i-1}
   \le 6c^2 X_{i+1}^{1-\lambda} X_i^{-\lambda},
\]
thus $X_i^\lambda \le 6c^2 X_{i+1}^{1-\lambda}$.  Combining this with \eqref{i:eq:Xip}, we deduce that $X_i^{\lambda^2} \le (6c^2)^\lambda(cc_1X_i)^{1-\lambda}$ for each $i\in I$, thus $\lambda^2\le 1-\lambda$ and so $\lambda\le 1/\gamma$.  Moreover, if $\lambda=1/\gamma$, this yields $1\le 6c^2(cc_1)^{1/\gamma}$, and so $c$ is bounded below by a positive constant depending only on $\varphi$ and $\uXi$.

%
%

\section{Proof of the third part of the main theorem}
\label{sec:iii}

The arguments in \cite[\S5]{RcubicI} can easily be adapted to show that, for some $\epsilon>0$ there are at most countably many irrational non-quadratic $\xi\in\bR$ with $\lambda(1:\xi:\xi^2)\ge 1/\gamma-\epsilon$.  This is, originally, an observation of S.~Fischler who, in unpublished work, also computed an explicit value for $\epsilon$.  The question was later revisited by D.~Zelo who showed in \cite[Cor.~1.4.7]{Z} that one can take $\epsilon=3.48\times 10^{-3}$, and who also proved a $p$-adic analog of this result.  More recently, the existence of such $\epsilon$ was established by P.~Bel, in a larger context where $\bQ$ is replaced by a number field $K$, and $\bR$ by a completion of $K$ at some place \cite[Theorem 1.3]{B}.  By Lemmas \ref{proj:lemma:T} and \ref{proj:lemma:reduction} (i), this proves Theorem \ref{proj:thm} (iii) when $\varphi$ is irreducible over $\bR$ and has a non-trivial zero in $\bP^2(\bQ)$.

We now consider the complementary case.  Using the notation and results of the previous section, we need to show that, when $\lambda$ is sufficiently close to $1/\gamma$, the point $\Xi$ lies in a countable subset of $\cC$.  For this purpose, we may assume that $\lambda>1/2$.  The next two lemmas introduce a polynomial $\psi(\ux,\uy)$ with both algebraic and numerical properties analog to that of the operator $[\ux,\ux,\uy]$ from \cite[\S2]{RcubicI} (cf. Lemmas 2.1 and 3.1(iii) of \cite{RcubicI}).

\begin{lemma}
\label{iii:lemma:identities:psi}
For any $\ux,\uy\in\bZ^3$, we define
\[
 \psi(\ux,\uy) := \Phi(\ux,\uy)\ux-\varphi(\ux)\uy \in\bZ^3
\]
Then, $\uz = \psi(\ux,\uy)$ satisfies $\varphi(\uz)=\varphi(\ux)^2\varphi(\uy)$ and $\psi(\ux,\uz)=\varphi(\ux)^2\uy$.
\end{lemma}

\begin{proof}
For any $a,b\in\bQ$, we have $\varphi(a\ux+b\uy)=a^2\varphi(\ux)+ab\Phi(\ux,\uy)+b^2\varphi(\uy)$.  Substituting $a=\Phi(\ux,\uy)$ and $b=-\varphi(\ux)$ in this equality yields $\varphi(\uz)=\varphi(\ux)^2\varphi(\uy)$.  The formula for $\psi(\ux,\uz)$ follows from the linearity of $\psi$ in its second argument.
\end{proof}

\begin{lemma}
\label{iii:lemma:estimates:psi}
Let $i,j\in\bZ$ with $i_0\le i<j$.  Then, the point $\uw = \psi(\ux_i,\ux_j)\in\bZ^3$ is non-zero and satisfies
\[
 \|\uw\| \ll X_i^2L_j + X_jL_i^2
 \et
 L(\uw) \ll X_j L_i^2.
\]
\end{lemma}

Here and for the rest of this section, the implied constants depend only on $\uXi$, $\varphi$, $\lambda$ and $c$.

\begin{proof}
Since $\ux_i$ and $\ux_j$ are distinct primitive elements of $\bZ^3$, they are linearly independent over $\bQ$. As $\varphi(\ux_i)\neq 0$, this implies that $\uw=\Phi(\ux_i,\ux_j)\ux_i-\varphi(\ux_i)\ux_j\neq 0$. By \eqref{i:eq:varphi}, we have
\[
 \varphi(\ux_i) = X_i\Phi(\uXi,\Delta_i) + \cO(L_i^2)
\]
where $\Delta_i=\ux_i-X_i\uXi$. Similarly, for $\Delta_j=\ux_j-X_j\uXi$, we find
\begin{align*}
 \Phi(\ux_i,\ux_j)
  = X_j\Phi(\uXi,\Delta_i) + X_i\Phi(\uXi,\Delta_j) + \Phi(\Delta_i,\Delta_j)
  = X_j\Phi(\uXi,\Delta_i) + \cO(X_iL_j).
\end{align*}
Substituting these expressions in the formula for $\uw=\psi(\ux_i,\ux_j)$, we obtain
\begin{align*}
 \uw
  &= \big(X_j\Phi(\uXi,\Delta_i) + \cO(X_iL_j)\big)(X_i\uXi+\Delta_i)
    - \big(X_i\Phi(\uXi,\Delta_i) + \cO(L_i^2)\big)(X_j\uXi+\Delta_j)\\
  &= \cO(X_i^2L_j+X_jL_i^2)\uXi + \cO(X_jL_i^2),
\end{align*}
and the conclusion follows.
\end{proof}

We will also need the following result, where the set $I$ (defined in Section \ref{sec:i}) is endowed with its natural ordering as a subset of $\bN$.

\begin{lemma}
\label{iii:lemma:estX}
For each triple of consecutive elements $i<j<k$ in $I$, the points $\ux_i$, $\ux_j$ and $\ux_k$ are linearly independent.  We have
\[
 X_j^\alpha \ll X_i \ll X_j^\theta \et L_i \ll X_j^{-\alpha}
 \quad\text{where}\quad
 \alpha=\frac{2\lambda-1}{1-\lambda}  \et  \theta=\frac{1-\lambda}\lambda.
\]
\end{lemma}

\begin{proof}
The fact that $i$ and $j$ are consecutive elements of $I$ implies that $\ux_i,\ux_{i+1},\dots,\ux_j$ belong to the same $2$-dimensional subspace $V_i = \langle\ux_i,\ux_{i+1}\rangle_\bR$ of $\bR^3$.  Similarly, $\ux_j,\ux_{j+1},\dots,\ux_k$ belong to $V_j = \langle\ux_j,\ux_{j+1}\rangle_\bR$.  Thus $\ux_i$, $\ux_j$ and $\ux_k$ span $V_i+V_j = \langle\ux_{j-1},\ux_j,\ux_{j+1}\rangle_\bR = \bR^3$, and so they are linearly independent.  Then, the normal vectors $\ux_i\wedge\ux_{i+1}$ to $V_i$ and $\ux_j\wedge\ux_{j+1}$ to $V_j$ are non-parallel and both orthogonal to $\ux_j$. So, their cross-product is a non-zero multiple of $\ux_j$.  Since $\ux_j$ is a primitive point of $\bZ^3$ and since these normal vectors have integer coordinates, their cross-product is more precisely a non-zero integer multiple of $\ux_j$.  This yields
\[
 X_j \le \|\ux_j\|
     \ll \|\ux_i\wedge\ux_{i+1}\|\,\|\ux_j\wedge\ux_{j+1}\|
     \ll (X_{i+1}L_i)(X_{j+1}L_j)
     \ll (X_{i+1}X_{j+1})^{1-\lambda}.
\]
If we use the trivial upper bounds $X_{i+1}\le X_j$ and $X_{j+1} \le X_k$ to eliminate $X_{i+1}$ and $X_{j+1}$ from the above estimate, we obtain $X_j \ll X_k^\theta$.  If instead we use the upper bounds $X_{i+1}\ll X_i^{1/\lambda}$ and $X_{j+1}\ll X_j^{1/\lambda}$ coming from \eqref{i:eq:Xip}, we find instead $X_j^\alpha\ll X_i$.  Finally, if we only eliminate $X_{j+1}$ using $X_{j+1}\ll X_j^{1/\lambda}$, we obtain $X_j^{\alpha/\lambda}\ll X_{i+1}$ and thus $L_i \ll X_{i+1}^{-\lambda}\ll X_j^{-\alpha}$.
\end{proof}

\begin{prop}
\label{iii:prop:y}
Suppose that $\lambda\ge 0.613$ and, for each integer $k\ge 1$, put $\uy_k=\ux_{i_k}$ where $i_k$ is the $k$-th element of $I$.  Then, for each sufficiently large $k$, the point $\uy_{k+1}$ is a rational multiple of $\psi(\uy_{k},\uy_{k-2})$.
\end{prop}

\begin{proof}
For each integer $k\ge 1$, let $Y_k$ denote the first coordinate of $\uy_k$.  Then, according to Lemma \ref{iii:lemma:estX}, we have $Y_{k+1}^\alpha \ll Y_k \ll Y_{k+1}^\theta$ and $L(\uy_k) \ll Y_{k+1}^{-\alpha}$,
with $\alpha\ge 0.5839$ and $\theta\le 0.6314$.  Put $\uw_k=\psi(\uy_k,\uy_{k+1})$.  By Lemma \ref{iii:lemma:estimates:psi}, the point $\uw_k$ is non-zero, and the above estimates yield
\[
 L(\uw_k) \ll Y_{k+1} L(\uy_k)^2 \ll Y_{k+1}^{1-2\alpha}
 \et
 \|\uw_k\| \ll Y_k^2 L(\uy_{k+1}) \ll Y_{k+2}^{-\alpha}Y_k^2
\]
(we dropped the term $Y_{k+1} L(\uy_k)^2$ in the upper bound for $\|\uw_k\|$ because it tends to $0$ as $k\to\infty$ while $\|\uw_k\|\ge 1$).  Using these estimates, we find
\begin{align*}
 |\det(\uy_{k-2},\uy_{k-1},\uw_k)|
     &\ll \|\uw_k\|L(\uy_{k-2})L(\uy_{k-1}) + \|\uy_{k-1}\|L(\uy_{k-2})L(\uw_k)\\
     &\ll Y_{k+2}^{-\alpha}Y_k^{2-\alpha^2-\alpha} + Y_{k-1}^{1-\alpha}Y_{k+1}^{1-2\alpha},\\
     &\ll Y_{k+2}^{-\alpha+\theta^2(2-\alpha^2-\alpha)} + Y_{k+1}^{\theta^2(1-\alpha)+1-2\alpha},\\
  |\det(\uy_{k-3},\uy_{k-2},\uw_k)|
     &\ll \|\uw_k\|L(\uy_{k-3})L(\uy_{k-2}) + \|\uy_{k-2}\|L(\uy_{k-3})L(\uw_k)\\
     &\ll Y_{k+2}^{-\alpha}Y_k^{2-\alpha^3-\alpha^2} + Y_{k-2}^{1-\alpha}Y_{k+1}^{1-2\alpha},\\
     &\ll Y_{k+2}^{-\alpha+\theta^2(2-\alpha^3-\alpha^2)} + Y_{k+1}^{\theta^3(1-\alpha)+1-2\alpha}.
\end{align*}
Thus both determinants tend to $0$ as $k\to\infty$ and so, for each sufficiently large $k$, they vanish.  Since, by Lemma \ref{iii:lemma:estX}, $\uy_{k-3},\uy_{k-2},\uy_{k-1}$ are linearly independent, this implies that, for those $k$, the point $\uw_k$ is a rational multiple of $\uy_{k-2}$.  As Lemma \ref{iii:lemma:identities:psi} gives $\psi(\uy_k,\uw_k)=\varphi(\uy_k)^2\uy_{k+1}$, we conclude that $\uy_{k+1}$ is a rational multiple of $\psi(\uy_k, \uy_{k-2})$ for each large enough $k$.
\end{proof}

We end this section with two corollaries.  The first one gathers properties of the sequence $(\uy_k)_{k\ge 1}$ when $\lambda=1/\gamma$.  The second completes the proof of Theorem \ref{proj:thm}(iii).

\begin{cor}
\label{iii:cor:y}
Suppose that $\lambda=1/\gamma$.  Then, the sequence $(\uy_k)_{k\ge 1}$ consists of primitive points of $\bZ^3$ such that $\psi(\uy_k,\uy_{k-2})$ is an integer multiple of $\uy_{k+1}$ for each sufficiently large $k$.  Any three consecutive points of this sequence are linearly independent and, for each $k\ge 1$, we have $\|\uy_{k+1}\|\asymp \|\uy_k\|^\gamma$, $L(\uy_k)\asymp \|\uy_k\|^{-1}$ and $|\varphi(\uy_k)| \asymp 1$.
\end{cor}

\begin{proof}
The first assertion simply adds a precision on Proposition \ref{iii:prop:y} based on the fact that $\uy_{k+1}$ is a primitive integer point.  Aside from the estimate for $|\varphi(\uy_k)|$, the second assertion is a direct consequence of Lemma \ref{iii:lemma:estX} since, for $\lambda=1/\gamma$, we have $\alpha=\theta=1/\gamma$.  To complete the proof, we use the estimate $|\varphi(\ux_i)|\ll X_iL_i$ established in the previous section as a consequence of \eqref{i:eq:varphi}.  Since $\varphi(\uy_k)$ is a non-zero integer, it yields $1\le |\varphi(\uy_k)|\ll 1$.
\end{proof}

\begin{cor}
\label{iii:cor:countable}
Suppose that $\lambda\ge 0.613$.  Then, $\Xi$ belongs to a countable subset of\/ $\cC$.
\end{cor}

\begin{proof}
Since each $\uy_k$ is a primitive point of $\bZ^3$ with positive first coordinate, the proposition shows that the sequence $(\uy_k)_{k\ge 1}$ is uniquely determined by its first terms.  As there are countably many finite sequences of elements of $\bZ^3$ and as the image of $(\uy_k)_{k\ge 1}$ in $\bP^2(\bR)$ converges to $\Xi$, the point $\Xi$ belongs to a countable subset of $\cC$.
\end{proof}

%
%

\section{Proof of the second part of the main theorem}
\label{sec:ii}

By \cite[Theorem 1.1]{RcubicI}, there exist countably many irrational non-quadratic real numbers $\xi$ for which $1/\gamma$ is an exponent of approximation to $(1:\xi:\xi^2)$.  Thus Part (ii) of Theorem \ref{proj:thm} holds for $\varphi = x_0x_2-x_1^2$ and consequently, by Lemmas \ref{proj:lemma:T} and \ref{proj:lemma:reduction}, it holds for any quadratic form $\varphi\in\bQ[x_0,x_1,x_2]$ which is irreducible over $\bR$ and admits at least one zero in $\bP^2(\bQ)$. These lemmas also show that, in order to complete the proof of Theorem \ref{proj:thm}(ii), we may restrict to a diagonal form $\varphi=x_0^2-bx_1^2-cx_2^2$ where $b>1$ is a square free integer and where $c$ is either $0$ or a square-free integer with $c>1$.  In fact, this even covers the case of $\varphi = x_0x_2-x_1^2$ since $(x_0+x_1+x_2)(x_0-x_1-x_2)-(x_1-x_2)^2=x_0^2-2x_1^2-2x_2^2$.

We first establish four lemmas which apply to any quadratic form $\varphi\in\bQ[x_0,x_1,x_2]$ and its associated  symmetric bilinear form $\Phi$ with $\Phi(\ux,\ux)=2\varphi(\ux)$.  Our first goal is to construct sequences $(\uy_i)$ as in Corollary \ref{iii:cor:y}.  On the algebraic side, we first make the following observation.

\begin{lemma}
\label{ii:lemma:alg}
Suppose that $\uy_{-1},\uy_0,\uy_1\in\bZ^3$ satisfy $\varphi(\uy_i)=1$ for $i=-1,0,1$.  We extend this triple to a sequence $(\uy_i)_{i\ge -1}$ in $\bZ^3$ by defining recursively $\uy_{i+1}=\psi(\uy_i,\uy_{i-2})$ for each $i\ge 1$.  We also define $t_i=\Phi(\uy_{i+1},\uy_i)\in\bZ$ for each $i\ge -1$.  Then, for any integer $i\ge 1$, we have
\begin{enumerate}
 \item[(a)] $\varphi(\uy_{i-2})=1$,
 \item[(b)] $\det(\uy_i,\uy_{i-1},\uy_{i-2})=(-1)^{i-1}\det(\uy_1,\uy_0,\uy_{-1})$,
 \item[(c)] $t_i=\Phi(\uy_{i+1},\uy_i)=\Phi(\uy_i,\uy_{i-2})$,
 \item[(d)] $\uy_{i+1}=t_i\uy_i-\uy_{i-2}$,
 \item[(e)] $t_{i+1}=t_it_{i-1}-t_{i-2}$.
\end{enumerate}
In particular, $t_{-1}=\Phi(\uy_0,\uy_{-1})$, $t_0=\Phi(\uy_1,\uy_0)$ and $t_1=\Phi(\uy_1,\uy_{-1})$.
\end{lemma}

\begin{proof}
By Lemma \ref{iii:lemma:identities:psi}, we have $\varphi(\uy_{i+1})=\varphi(\uy_i)^2\varphi(\uy_{i-2})$ for each $i\ge 1$.  This yields (a) by recurrence on $i$.  Then, by definition of $\psi$, the recurrence formula for $\uy_{i+1}$ simplifies to
\begin{equation}
 \label{ii:lemma:alg:eq1}
 \uy_{i+1} = \Phi(\uy_i,\uy_{i-2})\uy_i - \uy_{i-2} \quad (i\ge 1),
\end{equation}
and so $\det(\uy_{i+1},\uy_i,\uy_{i-1})=-\det(\uy_i,\uy_{i-1},\uy_{i-2})$ for each $i\ge 1$, by multilinearity of the determinant. This proves (b) by recurrence on $i$.  From \eqref{ii:lemma:alg:eq1}, we deduce that
\[
 t_i=\Phi(\uy_{i+1},\uy_i)=\Phi(\uy_i,\uy_{i-2})\Phi(\uy_i,\uy_i)-\Phi(\uy_{i-2},\uy_i)=\Phi(\uy_i,\uy_{i-2})
 \quad (i\ge 1),
\]
which is (c).  Then (d) is just a rewriting of \eqref{ii:lemma:alg:eq1}.  Combining (c) and (d), we find
\[
 t_{i+1}=\Phi(\uy_{i+1},\uy_{i-1})=t_i\Phi(\uy_i,\uy_{i-1})-\Phi(\uy_{i-2},\uy_{i-1})=t_it_{i-1}-t_{i-2}
 \quad (i\ge 1),
\]
which is (e).  Finally, for formula given for $t_{-1}$ and $t_0$ are taken from the definition while the one for $t_1$ follows from (c).
\end{proof}

The next lemma provides mild conditions under which the norm of $\uy_i$ grows as expected.

\begin{lemma}
\label{ii:lemma:est}
With the notation of the previous lemma, suppose that $1\le t_{-1}<t_0<t_1$ and that $1 \le \|\uy_{-1}\| < \|\uy_0\| < \|\uy_1\|$.  Then, $(t_i)_{i\ge -1}$ and $(\|\uy_i\|)_{i\ge -1}$ are strictly increasing sequences of positive integers with $t_{i+1}\asymp t_i^\gamma$ and $\|\uy_{i+1}\| \asymp t_{i+2}\asymp \|\uy_i\|^\gamma$.
\end{lemma}

Here and below, the implied constants are simply meant to be independent of $i$.

\begin{proof}
Lemma \ref{ii:lemma:alg}(e) implies, by recurrence on $i$, that the sequence $(t_i)_{i\ge -1}$ is strictly increasing and, more precisely, that it satisfies
\begin{equation}
\label{ii:lemma:est:eq1}
 (t_i-1)t_{i-1} < t_{i+1} < t_it_{i-1} \quad (i\ge 1),
\end{equation}
which by \cite[Lemma 5.2]{Rexp} implies that $t_{i+1}\asymp t_i^\gamma$.  In turn, Lemma \ref{ii:lemma:alg}(d) implies, by recurrence on $i$, that the sequence $(\|\uy_i\|)_{i\ge -1}$ is strictly increasing with
\begin{equation}
\label{ii:lemma:est:eq2}
 (t_i-1)\|\uy_i\| < \|\uy_{i+1}\| < (t_i+1) \|\uy_i\| \quad (i\ge 1).
\end{equation}
Combining this with \eqref{ii:lemma:est:eq1}, we find that the ratios $\rho_i=\|\uy_i\|/t_{i+1}$ satisfy
\[
 (1-1/t_i)\rho_i \le \rho_{i+1} \le \frac{1+1/t_i}{1-1/t_{i+1}}\rho_i \le \frac{1}{(1-1/t_i)^2}\rho_i\quad (i\ge 1),
\]
and so $\rho_1c_1\le \rho_i \le \rho_1/c_1^2$ for each $i\ge 1$ where $c_1=\prod_{i\ge 1}(1-1/t_i)>0$ is a converging infinite product because $t_i$ tends to infinity with $i$ faster than any geometric series.  This means that $\rho_i\asymp 1$, thus $\|\uy_i\|\asymp t_{i+1}$, and so $\|\uy_{i+1}\| \asymp t_{i+2} \asymp \|\uy_i\|^\gamma$ because $t_{i+2}\asymp t_{i+1}^\gamma$.
\end{proof}

For any $\ux,\uy\in\bR^3$, we denote by $\langle\ux,\uy\rangle$ their standard scalar product.  When $\ux\neq 0$ and $\uy\neq 0$, we also denote by $[\ux]$, $[\uy]$ their respective classes in $\bP^2(\bR)$, and define the \emph{projective distance} between these classes by
\[
 \dist([\ux],[\uy])=\frac{\|\ux\wedge\uy\|}{\|\ux\|\,\|\uy\|}.
\]
It is not strictly speaking a distance on $\bP^2(\bR)$ but it behaves almost like a distance since it satisfies
\[
 \dist( [\ux], [\uz] ) \le \dist( [\ux], [\uy] ) + 2\, \dist( [\uy], [\uz] )
\]
for any non-zero $\uz\in\bR^3$ (see \cite[\S2]{Rexp}).  Moreover, the open balls for the projective distance form a basis of the usual topology on $\bP^2(\bR)$.  We can now prove the following result.

\begin{lemma}
\label{ii:lemma:Xi}
With the notation and hypotheses of Lemmas \ref{ii:lemma:alg} and \ref{ii:lemma:est}, suppose that $\uy_{-1}$, $\uy_0$ and $\uy_1$ are linearly independent.  Then there exists a zero $\uXi=(1,\xi_1,\xi_2)$ of $\varphi$ in $\bR^3$ with $\bQ$-linearly independent coordinates such that $\|\uXi\wedge\uy_i\| \asymp \|\uy_i\|^{-1}$ for each $i\ge 1$.  Moreover, $1/\gamma$ is an exponent of approximation to the corresponding point $\Xi=(1:\xi_1:\xi_2) \in \bP^2(\bR)$.
\end{lemma}

\begin{proof}
Our first goal is to show that $([\uy_i])_{i\ge 1}$ is a Cauchy sequence in $\bP^2(\bR)$ with respect to the projective distance.  To this end, we use freely the estimates of the previous lemma and define $\uz_i = \uy_i\wedge\uy_{i+1}$ for each $i\ge 1$.  By Lemma \ref{ii:lemma:alg}(b), the points $\uy_{i-1}$, $\uy_i$ and $\uy_{i+1}$ are linearly independent for each $i\ge 0$.  Thus, none of the products $\uz_i$ vanish, and so their norm is at least $1$.  Moreover, Lemma \ref{ii:lemma:alg}(d) applied first to $\uy_{i+1}$ and then to $\uy_i$ yields
\begin{equation}
 \label{ii:lemma:Xi:eq1}
 \uz_i = \uy_{i-2}\wedge\uy_i
       = t_{i-1}\uy_{i-2}\wedge\uy_{i-1}-\uy_{i-2}\wedge\uy_{i-3}
       = t_{i-1}\uz_{i-2}+\uz_{i-3}.
\end{equation}
The above equality $\uz_i = \uy_{i-2}\wedge\uy_i$ with $i$ replaced by $i-3$ implies that
\[
 \|\uz_{i-3}\|
   \le 2\|\uy_{i-5}\|\,\|\uy_{i-3}\|
   \ll t_{i-4}t_{i-2}
   \asymp t_{i-1}t_{i-5}^{-1}
   \le t_{i-1}t_{i-5}^{-1} \|\uz_{i-2}\|.
\]
In view of \eqref{ii:lemma:Xi:eq1}, this means that $\|\uz_i\| = t_{i-1} (1+\cO(t_{i-5}^{-1}))\, \|\uz_{i-2}\|$,  and thus
\[
 \frac{\|\uz_i\|}{t_i}
   = \frac{t_{i-1}t_{i-2}}{t_i} (1+\cO(t_{i-5}^{-1})) \frac{\|\uz_{i-2}\|}{t_{i-2}}
   = (1+\cO(t_{i-5}^{-1})) \frac{\|\uz_{i-2}\|}{t_{i-2}}
\]
since, by Lemma \ref{ii:lemma:alg}(e), we have $t_{i-1}t_{i-2}=t_i(1+t_{i-3}t_i^{-1})=t_i(1+O(t_{i-5}^{-1}))$.  As the series $\sum_{i\ge 1} t_i^{-1}$ converges, the same is true of the infinite products $\prod_{i\ge i_0}(1+ct_i^{-1})$ for any $c\in\bR$.  Thus the above estimates implies that $\|\uz_i\| \asymp t_i$, and so we find
\[
 \dist([\uy_i],[\uy_{i+1}])
   = \frac{\|\uz_i\|}{\|\uy_i\|\,\|\uy_{i+1}\|}
   \asymp \frac{t_i}{t_{i+1}t_{i+2}}
   \asymp t^{-2}_{i+1}
   \asymp \|\uy_i\|^{-2}.
\]
As the series $\sum_{i\ge 1} 2^i t_{i+1}^{-2}$ is convergent, we deduce that $([\uy_i])_{i\ge 1}$ forms a Cauchy sequence in $\bP^2(\bR)$, and that its limit $\Xi\in\bP^2(\bR)$ satisfies $\dist([\uy_i],\Xi) \asymp \|\uy_i\|^{-2}$.  In terms of a representative $\uXi$ of $\Xi$ in $\bR^3$, this means that
\begin{equation}
 \label{ii:lemma:Xi:eq3}
 \|\uy_i\wedge\uXi\| \asymp \|\uy_i\|^{-1}.
\end{equation}
To prove that $\uXi$ has $\bQ$-linearly independent coordinates, we use the fact that
\[
 \|\langle\uu, \uy_i\rangle\uXi - \langle\uu, \uXi\rangle\uy_i\| \le 2\|\uu\|\,\|\uy_i\wedge\uXi\|
\]
for any $\uu\in\bR^3$ \cite[Lemma 2.2]{Rexp}.  So, if $\langle \uu, \uXi \rangle = 0$ for some $\uu\in\bZ^3$, then, by \eqref{ii:lemma:Xi:eq3}, we obtain $|\langle\uu, \uy_i\rangle| \ll \|\uy_i\|^{-1}$ for all $i$. Then, as $\langle\uu, \uy_i\rangle$ is an integer, it vanishes for each sufficiently large $i$, and so $\uu=0$ because any three consecutive $\uy_i$ span $\bR^3$.  This proves our claim.  In particular, the first coordinate of $\uXi$ is non-zero, and we may normalize $\uXi$ so that it is $1$.  Then, as $i$ goes to infinity, the points $\|\uy_i\|^{-1}\uy_i$ converge to $\|\uXi\|^{-1}\uXi$ in $\bR^3$ and, since $\varphi(\|\uy_i\|^{-1}\uy_i) = \|\uy_i\|^{-2}$ tends to $0$, we deduce that $\varphi(\uXi)=0$.  Finally, $1/\gamma$ is an exponent of approximation to $\Xi$ because, for each $X\ge \|\uy_1\|$, there exists an index $i\ge 1$ such that $\|\uy_i\| \le X \le \|\uy_{i+1}\|$ and then, by \eqref{ii:lemma:Xi:eq3}, the point $\ux:=\uy_i$ satisfies both
\[
 \|\ux\|\le X
 \et
 \|\ux\wedge\uXi\| \asymp \|\uy_i\|^{-1} \asymp \|\uy_{i+1}\|^{-1/\gamma} \le X^{-1/\gamma}.
 \qedhere
\]
\end{proof}

The last lemma below will enable us to show that the above process leads to infinitely many limit points $\Xi$.

\begin{lemma}
\label{ii:lemma:unicity}
Suppose that $(\uy_i)_{i\ge -1}$ and $(\uy'_i)_{i\ge -1}$ are constructed as in Lemma \ref{ii:lemma:alg} and that both of them satisfy the hypotheses of the three preceding lemmas.  Suppose moreover that their images in $\bP^2(\bR)$ have the same limit $\Xi$.  Then there exists an integer $a$ such that $\uy'_i=\pm\uy_{i+a}$ for each $i\ge \max\{-1,-1-a\}$.
\end{lemma}

\begin{proof}
Let $\uXi=(1,\xi_1,\xi_2)$ be a representative of $\Xi$ in $\bR^3$, and for each $\ux\in\bZ^3$ define $L(\ux)$ as in \eqref{i:eq:L}.  The estimates of Lemma \ref{ii:lemma:Xi} imply that $L(\uy_i)\asymp
\|\uy_i\|^{-1}$ and $L(\uy'_i)\asymp \|\uy'_i\|^{-1}$.  For each sufficiently large index $j$, we can find an integer $i\ge 2$ such that $\|\uy_{i-1}\|^{3/2} \le \|\uy'_j\| \le \|\uy_i\|^{3/2}$ and the standard estimates yield
\begin{align*}
 |\det(\uy_{i-1},\uy_i,\uy'_j)|
 &\ll \|\uy'_j\|L(\uy_i)L(\uy_{i-1})+\|\uy_{i}\|L(\uy_{i-1})L(\uy'_j) \\
 &\ll \|\uy_i\|^{3/2}\|\uy_i\|^{-1}\|\uy_{i-1}\|^{-1}+\|\uy_{i}\|\,\|\uy_{i-1}\|^{-1}\|\uy_{i-1}\|^{-3/2}\\
 &\ll \|\uy_i\|^{1/2-1/\gamma} = o(1),
\end{align*}
and similarly $|\det(\uy_i,\uy_{i+1},\uy'_j)|\ll \|\uy_i\|^{-1/(2\gamma)}=o(1)$.  Thus, both determinants vanish when $j$ is large enough and then $\uy'_j$ is a rational multiple of $\uy_i$.  However, both points are primitive elements of $\bZ^3$ since $\varphi$ takes value $1$ on each of them.  So, we must have $\uy'_j=\pm \uy_i$.  Since the two sequences have the same type of growth, we conclude that there exist integers $a$ and $i_0\ge \max\{-1,-1-a\}$ such that $\uy'_i=\pm \uy_{i+a}$ for each $i\ge i_0$.  Choose $i_0$ smallest with this property.  If $i_0\ge \max\{0,-a\}$, then, using Lemma \ref{iii:lemma:identities:psi}, we obtain
\[
 \uy'_{i_0-1}
 = \psi(\uy'_{i_0+1},\uy'_{i_0+2})
 = \psi(\pm\uy_{i_0+1+a},\pm\uy_{i_0+2+a})
 = \pm \uy_{i_0-1+a}
\]
in contradiction with the choice of $i_0$.  Thus we must have $i_0=\max\{-1,-1-a\}$.
\end{proof}

In view of the remarks made at the beginning of this section, the last result below completes the proof of Theorem \ref{proj:thm}(ii).

\begin{prop}
Let $b>1$ be a square-free integer and let $c$ be either $0$ or a square-free integer with $c>1$.  Then the quadratic form $\varphi=x_0^2-bx_1^2-cx_2^2$ admits infinitely many zeros in $\bP^2(\bR)$ which have  $\bQ$-linearly independent homogeneous coordinates and for which $1/\gamma$ is an exponent of approximation.
\end{prop}

\begin{proof}
The Pell equation $x_0^2-bx_1^2=1$ admits infinitely many solutions in positive integers.  We choose one such solution $(x_0,x_1)=(m,n)$.  For the other solutions $(m',n')\in(\bN^*)^2$, the quantity $mm'-bnn'$ behaves asymptotically like $m'/(m+n\sqrt{b})$ as $m'\to\infty$ and thus, we have $m<mm'-bnn'<m'$ as soon as $m'$ is large enough.  We fix such a solution $(m',n')$.  We also choose a pair of integers $r,t>0$ such that $r^2-ct^2=1$.  Then, the three points
\[
 \uy_{-1}=(1,0,0), \quad \uy_0=(m,n,0) \et \uy_1=(rm',rn',t)
\]
are $\bQ$-linearly independent.  They satisfy
\[
 \|\uy_{-1}\|=1 < \|\uy_0\|=m < rm' \le \|\uy_1\|
 \et
 \varphi(\uy_i)=1 \quad (i=-1,0,1).
\]
For such a triple, consider the corresponding sequences $(t_i)_{i\ge -1}$ and $(\uy_i)_{i\ge -1}$ as defined in Lemma \ref{ii:lemma:alg}.  The symmetric bilinear form attached to $\varphi$ being $\Phi=2(x_0y_0-bx_1y_1-cx_2y_2)$, we find
\[
 t_{-1}=2m < t_0=2r(mm'-bnn') < t_1=2rm'.
\]
Therefore the hypotheses of Lemmas \ref{ii:lemma:est} and \ref{ii:lemma:Xi} are fulfilled and so the sequence $([\uy_i])_{i\ge -1}$ converges in $\bP^2(\bR)$ to a zero $\Xi$ of $\varphi$ which has $\bQ$-linearly independent homogeneous coordinates and for which $1/\gamma$ is an exponent of approximation. To complete the proof and show that there are infinitely many such points, it suffices to prove that any other choice of $m,n,m',n',r,t$ as above leads to a different limit point.  Clearly, it leads to a different sequence $(\uy'_i)_{i\ge -1}$.  If $[\uy'_i]$ and $[\uy_i]$ converge to the same point $\Xi$ as $i\to\infty$, then by Lemma \ref{ii:lemma:unicity}, there exists $a\in\bZ$ such that $\uy'_i=\pm \uy_{i+a}$ for each $i\ge \max\{-1,-1-a\}$.  But, in both sequences $(\uy_i)_{i\ge -1}$ and $(\uy'_i)_{i\ge -1}$, the first point is the only one of norm $1$, and moreover the first three points have non-negative entries. So, we must have $a=0$ and $\uy'_i=\uy_i$ for $i=-1,0,1$, a contradiction.
\end{proof}

%
%


\begin{thebibliography}{99}

%
\bibitem{B}
  P.~Bel,
  Approximation simultan\'ee d'un nombre v-adique et de son carr\'e
  par des nombres alg\'ebriques,
  \textit{J.~Number Theory},
  to appear.
%
\bibitem{BL}
  Y.~Bugeaud, M.~Laurent,
  Exponents of Diophantine approximation and Sturmian continued fractions,
  \textit{Ann.\ Inst.\ Fourier} \textbf{55} (2005),
  773-804.
%
\bibitem{DS}
  H.~Davenport, W.~M.~Schmidt,
  Approximation to real numbers by algebraic integers,
  \textit{Acta Arith.\ }\textbf{15} (1969),
  393--416.
%
\bibitem{K}
  D.~Kleinbock,
  Extremal subspaces and their submanifolds,
  \textit{Geom.\ Funct.\ Anal.\ }\textbf{13} (2003),
  437–466.
%
\bibitem{La}
  M.~Laurent,
  Simultaneous rational approximation to
  the successive powers of a real number,
  \textit{Indag.\ Math.\ (N.S.)} {\bf 11} (2003),
  45--53.
%
\bibitem{LR}
  S.~Lozier and D.~Roy,
  Simultaneous approximation to a real number and to its cube,
  submitted.
%
\bibitem{Rnote}
  D.~Roy,
  Approximation simultan\'ee d'un nombre et de son carr\'e,
  \textit{C.\ R.\ Acad.\ Sci., Paris, ser.\ I} {\bf 336}
  (2003), 1--6.
%
\bibitem{RcubicII}
  D.~Roy,
  Approximation to real numbers by cubic algebraic integers (II),
  \textit{Ann.\ of Math.\ }{\bf 158} (2003), 1081--1087.
%
\bibitem{RcubicI}
  D.~Roy, Approximation to real numbers by cubic algebraic integers I,
  \textit{Proc.\ London Math.\ Soc.\ }\textbf{88}
  (2004), 42--62.
%
\bibitem{Rexp}
  D.~Roy,
  On two exponents of approximation related to a real number and
  its square,
  \textit{Canad.\ J.\ Math.\ }{\bf 59}
  (2007), 211--224.
%
\bibitem{Racta}
  D.~Roy,
  On simultaneous rational approximations to a real number, its square, and its cube,
  \textit{Acta Arith.\ }\textbf{133} (2008), 185--197.
%
\bibitem{Sc}
  W.~M.~Schmidt,
  \textit{Diophantine approximation}, Lecture Notes in Math.,
  vol.~785, Sprin\-ger-Verlag, 1980.
%
\bibitem{Z}
  D.~Zelo,
  Simultaneous approximation to real and $p$-adic numbers,
  PhD.~thesis, University of Ottawa, 2009; arXiv:math.NT/0903.0086.
%
\end{thebibliography}
\end{document}